%% file: Projectively_Goodman.tex
\documentclass[11pt]{amsart}

\usepackage[T1]{fontenc}
\usepackage[utf8]{inputenc}
\usepackage{lmodern}
\usepackage{layout}
\usepackage[american]{babel}
\usepackage[babel, final]{microtype}
\usepackage{amssymb}
\usepackage{amscd}
\usepackage[mathscr]{eucal}
\usepackage{amsmath}
\usepackage{outlines}
\usepackage{enumitem}

\usepackage{float}

\usepackage[pdftex, dvipsnames]{xcolor}
\usepackage[pdftex, final]{graphicx}
\usepackage{caption}
\usepackage{subcaption}
\usepackage{pinlabel}
\usepackage[pdftex, letterpaper, includehead, includefoot, nomarginpar, 
margin=1.1in]{geometry}
\definecolor{mypurple}{RGB}{40,113,220}
\definecolor{myorange}{RGB}{255,60,0}
\usepackage[pdftex, final, colorlinks=true, urlcolor=myorange, 
linkcolor=myorange, citecolor=mypurple, filecolor=purple, menucolor=purple, 
bookmarks=true, bookmarksdepth=3, bookmarksnumbered=true, bookmarksopen=true, 
bookmarksopenlevel=2]{hyperref}
\hypersetup{
  pdftitle={Surgery on conformally Anosov flows},
  pdfauthor={Federico Salmoiraghi},
  pdfsubject={Surgery, Legendrian tranverse Knots, Anosov Flows, Contact 3-Manifolds},
  pdfkeywords={ Surgery, Legendrian tranverse Knots, Anosov Flows, Contact 3-Manifolds
    57M27, 57R58}
}
\usepackage{tikz}
\usetikzlibrary{arrows, automata}
\usetikzlibrary{cd}

\usepackage[final]{showlabels} 

\makeatletter
\g@addto@macro\@floatboxreset\centering
\makeatother

\allowdisplaybreaks[4]

\addto\extrasamerican{%
}
\let\fullref\autoref
\newtheorem{maintheorem}{Theorem}

\newtheorem{theorem}{Theorem}[section]
\newtheorem{corollary}{Corollary}[section]
\newtheorem{proposition}{Proposition}[section]
\newtheorem{lemma}{Lemma}[section]

\theoremstyle{definition}
\newtheorem{definition}{Definition}[section]
\newtheorem{remark}{Remark}[section]
\newtheorem{example}{Example}[section]

\newtheorem{question}{Question}[section]
\makeatletter
\let\c@maincorollary=\c@maintheorem
\let\c@corollary=\c@theorem
\let\c@proposition=\c@theorem
\let\c@lemma=\c@theorem
\let\c@remark=\c@theorem
\let\c@definition=\c@theorem
\let\c@example=\c@theorem
\let\c@question=\c@theorem
\makeatother
\def\makeautorefname#1#2{\expandafter\def\csname#1autorefname\endcsname{#2}}
\makeautorefname{maintheorem}{Theorem}%
\makeautorefname{maincorollary}{Corollary}%
\makeautorefname{theorem}{Theorem}%
\makeautorefname{corollary}{Corollary}%
\makeautorefname{proposition}{Proposition}%
\makeautorefname{lemma}{Lemma}%
\makeautorefname{remark}{Remark}%
\makeautorefname{definition}{Definition}%
\makeautorefname{example}{Example}%
\makeautorefname{equation}{equation}%

\begin{document}

%
\title[Goodman surgery and projectively Anosov flows]{Goodman surgery and projectively Anosov flows}

\author{Federico Salmoiraghi}\thanks{This research was supported by the Israel Science Foundation (grant No. 1504/18)}
\address{Department of Mathematics \\ Israel Institute of Technology \\ Haifa, Israel}
\email{\href{mailto:salmoiraghi@campus.technion.ac.il}{salmoiraghi@campus.technion.ac.il}}



\begin{abstract}

We introduce a generalization of Goodman surgery to the category of projectively Anosov flows. This construction is performed along a knot that is simultaneously Legendrian and transverse for a supporting bi-contact structure. If the flow is Anosov there is a particular class of supporting bi-contact structures that induce Lorentzian metrics satisfying Barbot's criterion of hyperbolicity. 
Foulon and Hasselblatt construct new contact Anosov flows by surgery from a geodesic flow. We generalize their result showing that in any contact Anosov flow there is a family of Legendrian knots that can be used to produce new contact Anosov flows by surgery.  
Outside of the realm of Anosov flows we generate new examples of projectively Anosov flows on hyperbolic 3-manifolds. These flows contain an invariant submanifold of genus $g>0$.  
We also give some application to contact geometry: we interpret the bi-contact surgery in terms of classic contact-Legendrian surgery and admissible-inadmissible transverse surgery and we deduce some (hyper)tightness result for contact and transverse surgeries.

\end{abstract}

\maketitle



\input{Testing.tex}

\section{Introduction} 
\label{sec:intro}

\input{Introduction.tex}

\bibliographystyle{alpha} 
\bibliography{Bibliography}

\end{document}

%% file: Testing.tex
%
%
%
%
%
%
%
%
%

%% file: Introduction.tex
A classic example of the beautiful
intertwine between hyperbolic dynamics, foliation theory, and contact geometry is given by an Anosov flow. Geometrically an Anosov flow is determined by two transverse invariant foliations with expanding and contracting behaviors \cite{Ano1}, \cite{Ano2}.  Mitsumatsu \cite{Mit} first noticed that every Anosov vector field also belongs to the intersection of a pair of transverse contact structures rotating towards each other. We call these pairs {\it bi-contact structures} and the associated flows {\it projectively Anosov flows \footnote{Projectively Anosov flows appear in the existing literature also under the name of flows with {\it dominated splitting} (\cite{Pu}, \cite{PuM}) and {\it conformally Anosov flows} \cite{ElTh}}} (shortly pA flows). Projectively Anosov flows define a bigger class of flows than Anosov flows do. Moreover it is possible to show that there is a multitude of $3$-manifolds that carry bi-contact structures but do not carry an Anosov flow.

Hozoori (\cite{Hoz}, \cite{Hoz2}) has further highlighted the role of pA flows as a link between Anosov dynamics and symplectic geometry. This point of view has been recently used by \cite{KLMM} to define Floer-type invariants for Anosov flows. 

Projectively Anosov flows have been extensively studied also from the point of view of hyperbolic dynamics (\cite{HM}, \cite{AF}, \cite{Pu}, \cite{PuM}) and Riemman geometry (\cite{BL}, \cite{BLP}, \cite{Hoz3}, \cite{Hoz4}) while connections to foliation theory and topology have been known for decades (\cite{Asa}, \cite{EtGh}).
 
The importance of pA flows makes the introduction of a surgery operation desirable. In \cite{FS1} we defined a Dehn-type surgery along a knot in a bi-contact structure defining a volume preserving Anosov flow. 
In this paper we define a surgery that can be performed in {\it any}  projectively Anosov flow. To this end we identify a pA flow with an underlying bi-contact structure $(\xi_-,\xi_+)$ and we call a knot $K$ {\it Legendrian-transverse} if its tangents belongs to $\xi_-$ and they are tranverse to $\xi_+$. From now on we assume $q$ to be an integer. 
\begin{maintheorem}
\label{thm2}
Given a Legendrian-transverse knot $K$ in a bi-contact structure (equivalently, a pA flow) on a $3$-manifold, there is a $(1,q)$-Dehn type surgery on a transverse annulus $C$ containing $K$, that yields bi-contact structures for infinitely many values of $q$. 
\end{maintheorem}

The bi-contact surgery of \fullref{thm2} can be interpreted as an extension to pA flows of the surgery operation on Anosov flows introduced by Goodman \cite{Goo}. Indeed, in a general Anosov flow the operation of \fullref{thm2} can be performed under similar hypothesis of Goodman's construction (see \cite{Hoz2}) and it generates the {\it same} flows.

\subsubsection{Applications to (contact) Anosov flows}In order to show the Anosovity of a flow it is enough to prove the existence of a cone field satisfying a number of conditions according to the {\it cone field criterion for hyperbolicity}. Goodman used this argument to show that we can produce new Anosov flows by surgery along a knot {\it near a closed orbit}. 

Barbot \cite{Bar1}  gives a formulation of the cone field criterion using Lorentzian metrics that are also {\it Lyapunov} in the sense that their positive cones are exponentially strictly contracting. Barbot's formulation of the cone criterion has been successfully used by Foulon and Hasselblatt to construct the first examples of {\it contact} Anosov flows on hyperbolic manifolds \cite{FoHa1} performing a Goodman-type surgery along a Legendrian knot in a geodesic flow \footnote{In contrast to Goodman's construction the knots used by Foulon and Hasselblatt are located far from a closed orbit.  Nevertheless, Foulon and Hasselblatt show that the surgered flow is Anosov for $q>0$. Their argument implicitly relies on particular features of the geodesic flow (see \fullref{AAA} for more details).}.

In a contact Anosov flow different from the geodesic  flow it is not clear along which Legendrian knots (if any) Foulon and Hasselblatt construction produces a new contact Anosov flow.
 Using \fullref{thm2} we show that in every contact Anosov flow there are families of Legendrian knots that can be used to generate new contact Anosov flows by surgery.

\begin{maintheorem}
\label{thm:7}
Let $\phi^t$ be any contact Anosov flow defined by the bi-contact structure $(\xi_-,\xi_+)$ and preserving a positive contact structure $\eta_+$. Consider a knot $L$ simultaneously Legendrian for $\eta_+$ and for $\xi_-$. The knot $L$ is Legendrian-transverse for $(\xi_-,\xi_+)$ and a $(1,q)$-bi-contact surgery as in \fullref{thm2} along $L$ yields a contact Anosov flow for every $q\in \mathbb{N}$.
\end{maintheorem}

One might wonder whether the Legendrian knots of \fullref{thm:7} are abundant. Note that a knot $L$ as in \fullref{thm:7} can be interpreted as a periodic orbit of a pA flow $\psi^t$ defined by the bi-contact structure $(\xi_-,\eta_+)$. 
In order to study the set of the closed orbits of $\psi^t$ we prove  the following.

\begin{proposition}
\label{thm:8}
Let $\phi^t$ be a contact Anosov flow defined by a bi-contact structure $(\xi_+,\xi_-)$ and preserving a contact structure $\eta_+$. The flow $\psi^t$ defined by the bi-contact structure $(\xi_-,\eta_+)$ is Anosov.

\end{proposition}

A consequence of \fullref{thm:8} is that, in a general contact Anosov flow, there are infinitely many Legendrian knots along which a $(1,q)$-surgery generates a new contact Anosov flow for every $q>
0$. 
\\
 
Furthermore, we use a particular family of contact forms supporting an Anosov flow to construct Lorentz-Lyapunov metrics and we express them in coordinates in concrete cases.

\begin{maintheorem}
\label{LL}
Let $\phi^t$ be an Anosov flow (transitive or not) with orientable weak foliations and generating vector field $X$. There is a pair of contact forms $(\alpha_-,\alpha_+)$ defining $X$ such that the pair of $C^0$ Lorentzian metrics 
$$Q^{\pm}=\pm \alpha_-\otimes \alpha_+-\beta^2$$
is Lyapunov.  
Here $\beta$ is the canonical invariant $1$-form such that $\ker\beta=E^{uu}\oplus E^{ss}$ and $ \beta(X)=1$.

\end{maintheorem} 

When the flow is volume preserving, a recent result of Hozoori \cite{Hoz2} shows that there is a is a reparametrization $X_L$ and a pair of contact forms $(\alpha_-,\alpha_+)$ supporting $X_L$ with Reeb vector fields $(R_{\alpha_-},R_{\alpha_+})$ such that $\beta$ is the  $1$-form satisfying $\ker\beta=\langle R_{\alpha_-} ,R_{\alpha_+} \rangle$ and $ \beta(X_L)=1$. This allows us to fully express these metrics in terms of a special pair of contact forms.

\subsubsection{Applications to contact geometry} An important question in contact geometry is to understand which properties of the contact structures are preserved by various flavours of surgery operations. A celebrated result of \cite{Wan1} states that a {\it Legendrian surgery} \footnote{Hystorically, a Legendrian surgery is a $(1,-1)$-surgery along a Legendrian knot. In general we call {\it contact} a surgery along a Legendrian knot in a contact structure.} in a {\it tight} contact structure produces a tight contact structure. In general a positive surgery along a Legendrian knot does not preserve tightness.

The bi-contact surgery operations defined in \fullref{thm2} have a natural interpretation in the framework of classic contact-Legendrian surgery (\cite{We}, \cite{Elia}, \cite{DiGe2}).  
Since Hozoori \cite{Hoz} proved that a bi-contact structure supporting an Anosov flow is {\it hypertight}, 
we have the following criterion for positive contact surgeries.
\begin{maintheorem}
\label{Cor2}
Let $\gamma$ be a closed orbit of any Anosov flow defined by the bi-contact structure $(\xi_-,\xi_+)$. For every $q\in \mathbb{N}$ a $(1,q)$-contact surgery along $\gamma$ yields an hypertight contact structure $\tilde{\xi}_+$.
\end{maintheorem}

Consider a $T^2$-bundle with Anosov monodromy equipped with a vertically rotating positive contact structure $\xi_+$ with minimal twisting.  A consequence of \fullref{Cor2} is that every $(1,q)$-contact surgery on $\xi_+$ along an orbit of the suspension flow is hypertight.\\

A surgery along a transverse knot in a contact structure is called {\it inadmissible} if it {\it adds twisting} to the original contact structure.  Such an operation does not produce in general, a tight contact structure. However, for bi contact structures defining a volume preserving Anosov flow we have the following.

\begin{maintheorem}
\label{Cor}
Let $K$ be a Legendrian-transverse knot (not necessarily isotopic to a closed orbit) in a bi-contact structure $(\xi_-,\xi_+)$ defining a volume preserving Anosov flow. Every inadmissible transverse $(1,q)$-surgery on $\xi_+$ along $K$ yields a hypertight contact structure $\tilde{\xi}_+$.
\end{maintheorem}

\subsubsection{Applications to pA flows } The bi-contact structures defining Anosov flows have been fully characterized from a symplectic and contact-geometric prospective by Hozoori \cite{Hoz}. 

In \cite{BBP} Bonatti, Bowden and Potrie sketch the construction of a projectively Anosov flow in a hyperbolic $3$-manifold that has attracting or repelling orbits (therefore not Anosov) with taut invariant foliations using an adaptation of the {\it hyperbolic plugs} as defined by B\'{e}guin, Bonatti and Yu in \cite{BBY}. 

Using \fullref{thm2} we construct new families of non-Anosov, bi-contact structures on hyperbolic manifolds.

\begin{maintheorem}\label{Non}
A $(1,q)$-Dehn surgery on the figure-eight knot supports infinitely many non-homotopic bi-contact structures. The associated pA-flows are not Anosov and contain an invariant surface of genus $g>0$. 
\end{maintheorem}

The construction of \fullref{Non} is roughly as follows. Mitsumatsu (\cite{Mit}, \cite{Mit2}) introduces an infinite family of tight bi-contact structures called {\it propellers} on every $T^2$-bundles over the circle. For instance, consider the Anosov automorphism of $ T^2$ induced by the linear map 
$$A=\begin{pmatrix}
2&1\\
1&1
\end{pmatrix}$$
and consider the torus bundle over the circle $M(0)=T^2\times [0,1]/\sim$, where the equivalence relation $\sim$ is defined by $(x,0)\sim(A(x),1)$. The complement of the knot $K$ given by the suspension of $(0,0)$ is diffeomorphic to the complement of the figure-eight knot in $S^3$. Call $M(q)$ the manifold obtained by performing Dehn $(1,q)$-surgery on $K$ in $M(0)$. As a consequence of a celebrated result of Thurston \cite{Th}, $M(q)$ is hyperbolic except when $q\in \{0,\pm1,\pm2,\pm3,\pm4\}$. It is easy to see that if $M(0)$ is equipped with a propeller, $K$ is always a Legendrian-transverse knot. Therefore, we can apply the surgery described in  \fullref{thm2} on $M(0)$ to produce $M(q)$ endowed with an infinite family of new bi-contact structures. 

The existence of annuli foliated by compact trajectories show that the new flows are not Anosov while the existence of an invariant submanifold of genus $g>1$ is a consequence of the existence of toroidal leaves in the invariant foliation of the initial pA flow.

\subsection{Aknowledgment}

The author would like to thank Tali Pinsky for her support and for introducing him to this fascinating subject. He would also like to thank Surena Hozoori and Martin Mion-Mouton for the helpful conversations. Finally the author would like to thank Thomas Barthelmé and Ilaria Patania for their advices that greatly improved the quality of the manuscript.

\section{Anosov flows, projectively Anosov flows and bi-contact structures} 
Anosov flows are a class of dynamical system characterized by the contracting and expanding behaviour of two transverse invariant directions 

\begin{definition}
Let $M$ be a closed manifold and $\phi^t:M \rightarrow M$ a $C^1$ flow on $M$. The flow $\phi^t$ is called {\it Anosov} if there is a splitting of the tangent biundle $TM=E^{uu} \oplus E^{ss} \oplus \langle X \rangle$ preserved by $D\phi^t$ and positive constants $A$ and $B$ such that
$$\lVert D \phi^t(v^u) \rVert\geq Ae^{Bt} \lVert  v^u\rVert\;\;\;\;\;\;\text{for any}\;v^u \in E^{uu}$$
$$\lVert D \phi^t(v^s) \rVert\leq Ae^{-Bt} \lVert  v^s\rVert\;\;\;\;\text{for any}\;v^s \in E^{ss}.$$
Here $\lVert \cdot \rVert$ is induce by a Riemmanian metric on $TM$.
We call $E^{uu}$ and $E^{ss}$ respectively the {\it strong unstable bundle} and the {\it strong stable bundle}. 

\end{definition}

Classic examples of Anosov flows are the geodesic flow on the unit tangent bundle of a hyperbolic surface and the suspension flows of hyperbolic linear automorphisms of the torus.

Anosov showed that the distributions $E^{ss}$ and $E^{uu}$ are uniquely integrable and the associated foliations are denoted by $\mathcal{F}^{ss}$ and $\mathcal{F}^{uu}$. Furthermore, the {\it weak stable bundle} $E^s=E^{ss} \oplus \langle X \rangle$ and the {\it weak unstable bundle} $E^u=E^{uu} \oplus \langle X \rangle$ are also uniquely integrable and the codimension one associated foliations are dented with $\mathcal{F}^{s}$ and $\mathcal{F}
^{u}$.

Mitsumatsu \cite{Mit} first noticed that an Anosov flow with orientable weak invariant foliations is tangent to the intersection of two transverse contact structures (see also Eliashberg-Thurston \cite{ElTh}). We will call such pairs {\it bi-contact structures}. However, the converse statement is not true and there are bi-contact structures that do not define Anosov flows as the following example shows.

 as showed in the following example.

\begin{example}
\label{ex:T3}
We construct a family of bi-contact structure on $T^3$ (called {\it propellers}) using a recipe introduced by Mitsumatsu in \cite{Mit} and \cite{Mit2}.
Consider the contact forms defined on $T^2\times I$
$$\alpha_n=cos(2n\pi z)dx-sin(2n\pi z)dy,$$
$$\alpha_{-m}=cos(2m\pi z)dx+sin(2m \pi z)dy.$$
They are not transverse to each on the tori defined by  $\{z=0\}$, $\{z=\frac{1}{4}\}$, $\{z=\frac{1}{2}\}$, $\{z=\frac{3}{4}\}$. If we introduce a perturbation $\epsilon(z)\: dz$ such that $\epsilon(z)$ is a function that does not vanish on these tori, the contact structures $\alpha_+=\alpha_n+\epsilon(z)\: dz$ and $\alpha_-=\alpha_{-m}$ are transverse (see \fullref{Propellers}). The presence of an compact invariant submanifolds shows that the associated flow is not Anosov.

\begin{figure}
\label{bundles}
\includegraphics[width=0.9\textwidth]{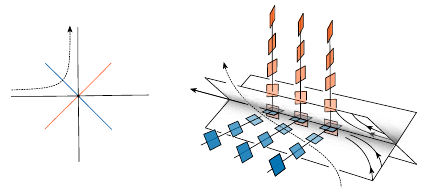}
  \caption{In red and blue the plane fields defining the bi-contact structure. }
\label{bundles}
\end{figure}

\end{example}
We have the following important generalization of Anosov flows. 
\begin{definition}
\label{pAf}
Let $M$ be a closed manifold and $\phi^t:M \rightarrow M$ a $C^1$ flow on $M$. The flow $\phi^t$ is called {\it projectively Anosov} if there is a splitting of the projectified tangent bundle $TM/ \langle X \rangle=\mathcal{E}^u \oplus \mathcal{E}^s$ preserved by $D\phi^t$ and positive constants $A$ and $B$ such that
$$\frac{\lVert D \phi^t(v^u) \rVert}{   \lVert D \phi^t(v^s) \rVert}\geq Ae^{Bt} \frac{\lVert  v^u\rVert}{\lVert  v^s\rVert}\;\;\;\;\;\;\;{\text for \; any}\;v^u \in \mathcal{E}^u \; {\text and}\;v^s \in \mathcal{E}^s $$

Here $\lVert \cdot \rVert$ is induce by a Riemmanian metric on $TM$.
We call $\mathcal{E}^u$ and $\mathcal{E}^s$ respectively the {\it unstable bundle} and the {\it stable bundle}. 

\end{definition}
The flows of \fullref{pAf} are referred also to as {\it conformally Anosov flows} and {\it flows with dominated splitting on $TM/ \langle X \rangle$}. 

The invariant bundles $\mathcal{E}^u$ and $\mathcal{E}^s$ induce invariant plane fields $E^u$ and $E^s$ on $M$. These plane fields are continuous and integrable, but unlike the Anosov case the integral manifolds my not be unique (see \cite{ElTh} and \cite{Nod}).
However, when they are smooth they also are uniquely integrable. We call these flows {\it regular projectively Anosov} (see \cite{Nod}, \cite{NoTs} and \cite{Asa} for a complete classification).

\begin{figure}
\label{Propellers}

\includegraphics[width=1\textwidth]{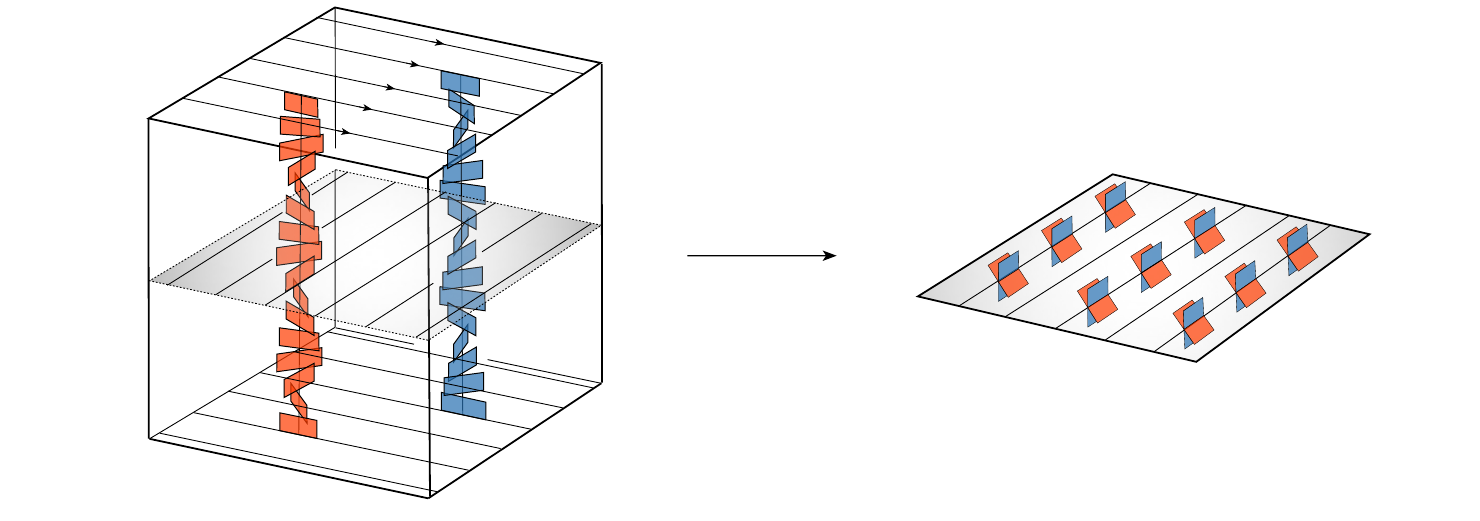}
  \caption{On left: a propeller on $T^3$. In red and blue the positive and negative contact structures. On the right:  adding a perturbation $\epsilon \:dz$ to one of the contact forms in a neighborhood of an horizontal torus $\Sigma$ along which the contact structures have collinear characteristic foliations, produces a bi-contact structure. The surface $\Sigma$ is a compact invariant submanifold foliated by flowlines. }
\label{Propellers}
\end{figure}

\begin{proposition}[Mitsumatsu \cite{Mit}]
Let $X$ be a $C^1$ vector field on $M$. Then $X$ is projectively Anosov if and only if it is defined by a bi-contact structure.
\end{proposition}

\begin{example}
\label{example1}

This construction can be extended to any $T^2$-bundle over $S^1$ if we make the directions of the linear characteristic foliations of $\xi=\ker \alpha$ and $\eta=\ker \beta$ on the top and bottom torus to agree with the rotation angles of the monodromy (\cite{Mit2} and see \fullref{Anosov}).  

 \begin{figure}
\label{Anosov}

\includegraphics[width=1\textwidth]{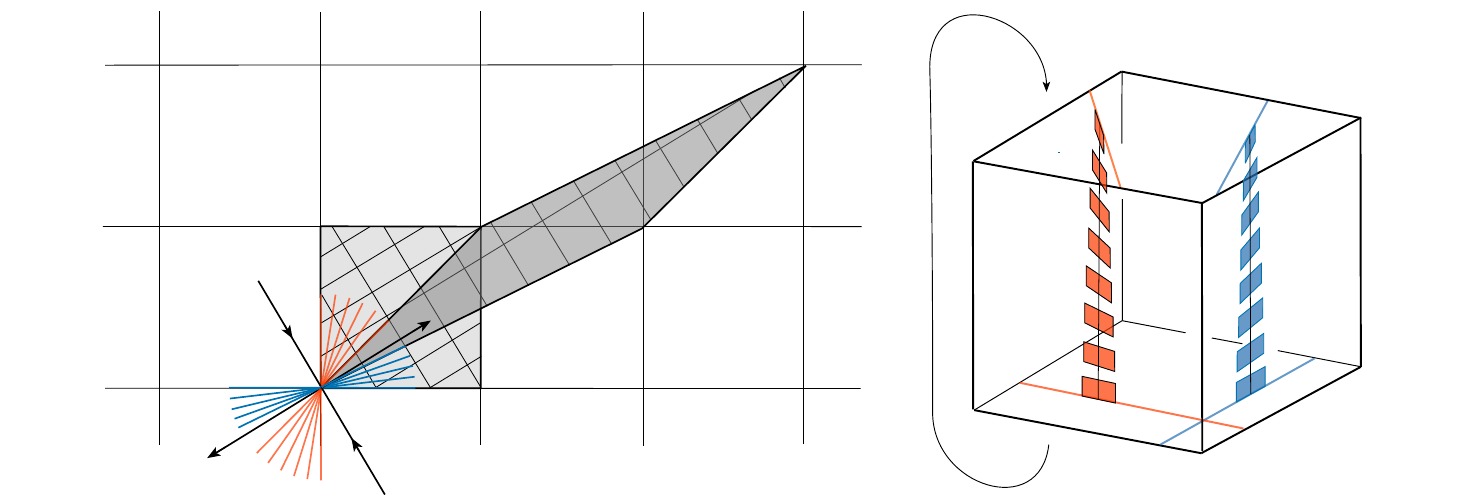}
  \caption{Propeller construction adapted to a $T^2$-bundle with linear Anosov monodromy. This example has minimal twisting, therefore it defines the Anosov suspension flow. Note that in this case no perturbation $\epsilon \: dz$ is needed.}
\label{Anosov}
\end{figure}
 
As Noda shows \cite{Nod}, all the $T^2$-bundles over $S^1$ with invariant foliations of class $C^2$ are obtained using this construction. Precisely, if the invariant foliations contains closed leaves, they are tori isotopic to linear flows on $T^2$ in the direction of the eigenvectors of the monodromy. 
 
In the case the $T^2$-bundle over $S^1$ has hyperbolic monodromy and the invariant foliations are of class $C^2$ without compact leaves, the flow is actually an Anosov flow given as the suspension of the monodromy. The bi-contact structures in these cases are obtained by propeller constructions with minimal twisting (see \fullref{Anosov}).
\end{example}

\subsection{Reeb dynamics of a bi-contact structure defining an Anosov flows}

We now present a characterization of bi-contact structures that define Anosov flows using the Reeb dynamics.

\begin{definition}[Hozoori \cite{Hoz}]
Consider a bi-contact structure $(\xi_-,\xi_+)$. A vector field is {\it dynamically positive (negative)} if at every $p\in M$ it lies in the interior of the first or third (second or forth) region in \fullref{bundles}.
\end{definition}

\begin{theorem}[Hozoori \cite{Hoz}]
\label{prop:Hoz}
Let $\phi^t$ be a projectively Anosov flow on $M$. The following are equivalent.
\end{theorem}

\begin{enumerate}
\item {The flow \it $\phi^t$ is Anosov},

\label{cond:h1}

\item {\it There is a pair} $(\alpha_-,\alpha_+)$ {\it of positive and negative contact forms defining} $\phi^t$ {\it such that the Reeb vector field of} $\alpha_+$ {\it is dynamically negative},

\label{cond:h2}

\item {\it There is a pair} $(\alpha_-,\alpha_+)$ {\it of positive and negative contact forms defining} $\phi^t$ {\it such that the Reeb vector field of} $\alpha_-$ {\it is dynamically positive}.

\label{cond:h3}
\end{enumerate}

\begin{remark}
We assume the reader familiar with contact geometry. We refer to \cite{Hoz} and \cite{Hoz2} for the basic notions in bi-contact geometry and projectively Anosov flows and for a precise discussion on the regularity of the plane fields, bundles and foliations involved. The bi-contact structures in the present work are defined by $C^1$ 1-forms. In particular a bi-contact structure $(\xi_-,\xi_+)$ is a pair of transverse plane fields $(\xi_-=\ker \alpha_-,\xi_+=\ker \alpha_+)$ for a pair of $C^1$ 1-forms $(\alpha_-,\alpha_+)$ such that $\alpha_+\wedge d\alpha_+>0$ is a positive $C^0$ volume form. Note that the Reeb  vector fields $R_{\alpha_-}$ and $R_{\alpha_+}$ are a priori just $C^0$.  A $C^{k-1}$ vector field $X$ defines a contact Anosov flow if $X$ is Anosov and is the Reeb vector field of a $C^{k}$ contact form $\beta$.  
\end{remark}

\section{Dehn type surgeries on $3$--manifolds}
\label{Dehn}
Surgery is a very effective way to produce new examples of manifold from classic ones. In this section we recall the basic notions and set the notation. We follow the descriptions and convention of \cite{Ge2} and \cite{BaEt1}. 

\label{3.1}
\subsection{Dehn surgery along a knot} We call {\it Dehn surgery along $K$} the following construction. We first remove a tubolar neighborhood $N$ of a knot $K$, then we glue a solid torus $S^1\times D^2$ using a diffeomorphism 
$\partial (S^1 \times D^2)\rightarrow \partial N$ that sends the meridian $\mu_0$ of $ \partial (S^1\times D^2)$ to the curve $p\mu+q\lambda$ in $\partial N$. Here $p$ and $q \in \mathbb{Z}$ are coprime, $\mu$ is the meridian of $\partial N$ while $\lambda$ is a preferred longitude (or framing of $K$) and we orient $\mu$ and $\lambda$ so that $\mu\cap \lambda=1$. We denote this operation as $(p,q)$-surgery. 

\subsection{Surgery along Legendrian knots}
If the ambient manifold is equipped with an additional geometric structure it is natural to ask for surgery operations that preserve it. For instance, given a contact $3$-manifold $(M,\xi_+)$ and a Legendrian knot $K$ in $M$, it is possible to define a Dehn-type surgery along $K$ that yield a new contact manifold $(\tilde{M},\tilde{\xi})$. We first choose the preferred longitude $\lambda$ given by the contact frame. This can be done considering the push--off of $K$ in a transverse direction to the contact planes along the knot. Every $K$ has a neighborhood $N_L$ contactomorphic to a standard one as follows. Consider on $\mathbb{R}^2 \times S^1$ the tight contact structure  $\xi_1=\ker(dt-w\:ds)$. Let $N^L$ the solid torus $\{(w,t,s):w^2+t^2\leq 1\}$. This is the neighborhood of the Legendrian curve $L=\{(w,t,s):w=w=0\}$ and its boundary  $\partial N^L$ is convex with two dividing curves parallel to $\Gamma=\{(w,t,s):w=0, t=1\}$. These solid tori are contactomorphic up to a perturbation of the boundary according to the following result
\begin{theorem}[Kanda \cite{Ka}]
\label{ka}
Let $\Gamma$ be a pair of longitudinal curves on the boundary of a solid torus $N$. Let $\mathcal{F}$ be a (singular) foliation on $\partial N$ that is divided by $\Gamma$. Then there is a unique (up to isotopy) tight contact structure on $N$ whose characteristic foliation on $\partial N$ is $\mathcal{F}$.
\end{theorem} 
We now remove a neighborhood $N$ of the Legendrian knot diffeomorphic to $N^L$ and glue back a solid torus $S^1\times D^2$ sending a meridian $\mu_0=*\times D^2$ where $*$ is a point of $S^1$ to a curve $p\mu+q\lambda$ on $\partial(\overline{M\setminus N})$. When we glue back the solid torus $S^1\times D^2$  we would like to extend the contact structure $\xi$ on $\overline{M\setminus N}$ over the solid torus. This boils down to finding contact structures on the solid torus $S^1\times D^2$ with a certain characteristic foliation on the convex boundary determined by the surgery coefficient $p/q$. When $p\neq 0$ this extension always exists and is tight in the solid torus by work of Honda  and Giroux. If $p=1$ (or equivalently for $(1,q)$-surgery), we can assume that the dividing set on the glued solid torus is a pair of curves parallel to the longitudinal curve $q\mu_0+\lambda_0$ where $\lambda_0=S^1\times *$ where $*$ is a point of $\partial D^2$. By \fullref{ka} there is a unique tight contact structure on $S^1\times D^2$ and therefore a unique extension in the new manifold. We call this Dehn type surgery {\it contact $(1,q)$-surgery}. Contact $(1,-1)$-surgery on $K$ is historically called {\it Legendrian surgery} on $K$.

\begin{theorem}[Wand \cite{Wan1}]
Legendrian surgery preserves tightness.
\end{theorem}
The results extends to every contact $(1,q)$-surgery with $q<0$.

\subsection{Surgery along transverse knots}
A knot $K$ in a contact manifold $(M,\xi_+)$ is {\it transverse} if its oriented tangent vector is positively transverse to $\xi_+$. The natural framework to define transverse surgery has been set up by Gay in \cite{Gay}, Conway in \cite{Con1} , and Baldwin and Etnyre in \cite{BaEt1}. This operation comes in two flavours, {\it admissible transverse surgery} and {\it inadmissible transverse surgery} both generalization of classic constructions introduced by Lutz and Martinet. Roughly speaking the first can be thought as removing twisting near a knot while the second as adding twisting.  

A transverse knot has a neighborhood $N_T$ contactomorphic to a solid torus $S^1\times D_{\{r<R\}}$ in $S^1\times \mathbb{R}^2$ with tight contact structure $\xi_{rot}=\ker(\cos f(r) ds+f(r) \sin f(r) d\theta)$ where $f:[0,\infty)\rightarrow [0,\pi)$ is an increasing surjective function of $r$ (see \fullref{transverse} ). If we identify some framing of $K$ with $\lambda=S^1\times \{r=R, \theta=0\}$. The characteristic foliation on the torus $r=R$ is given by parallel lines $f(R)\lambda-(\cot f(R)) \mu$ or equivalently lines of slope $k=-(\cot f(R))/f(R))$ where $\mu$ is a meridian of $K$. We will use a different sign convention to the one usually adopted by contact geometers. We associate the slope $0$ to the contact planes transverse to the core of the torus (identified with $K$) while the $\infty$ slope is associated with the preferred longitude $\lambda$. With this convention if we call $N^T_k$ the solid torus with characteristic foliation of slope $k\in (-\infty,\infty)$ on $\partial N^T_k$ we have $N^T_{k'}\subset N^T_k$ if $k'<k$. Note that different choices of the preferred longitude will give a different value of $k$.

\begin{figure}
\label{transverse}
\includegraphics[width=0.9\textwidth]{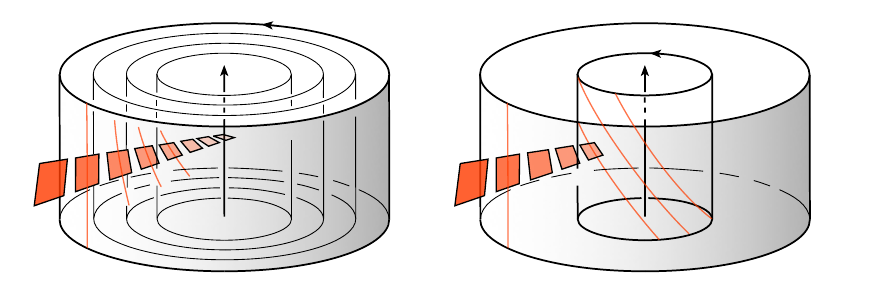}
  \caption{On the left: standard neighborhood of a transverse knot. On the right: collapsing the curves gives a (in this case negative and admissible) transverse surgery.}
\label{transverse}
\end{figure}

\subsubsection{Admissible transverse surgery} Remove a solid torus $N^T_{k'}$ inside $N^T_k$ and perform a {\it contact cut} on the boundary that consists in taking the quotient of the torus boundary $\partial (\overline{M\setminus N^T_{k'}})$ with respect to the leaves of the characteristic foliation on $\partial (\overline{M\setminus N^T_{k'}})$.
The quotient space $M'$ inherits a natural contact structure $\xi'$ and the set of points $K'\in M'$ with more than one pre-image under the quotient map $M\rightarrow M'$ is a transverse knot in the contact manifold $(M',\xi')$. 
\subsubsection{Inadmissible transverse surgery}
Consider the open manifold obtained by removing $K$ and than take the closure to obtain  a manifold with boundary with a contact structure that induces on $\partial M$ a characteristic foliation given by curves of slope $0$, or equivalently, meridional slope. We now glue a $T^2\times I$ with rotational contact structure such that the contact planes rotate $n$ half twists to some slope $k$, finally we perform a contact cut.

\subsection{Dehn type surgery on Anosov flows}

\label{G}

In this section we describe two Dehn-type surgery on Anosov flows. The first has been introduced by Goodman \cite{Goo}. This construction, inspired by Handel-Thurston one \cite{HaTh}, is performed along a transverse annulus sufficiently close to a closed orbit. We refer to \cite{Sha} for a more detailed description.
The second construction has been discovered by Foulon and Hasselblatt and produces new contact Anosov flows when it is performed along particular Legendrian knots in a geodesic flow.   

\subsubsection{Goodman surgery}
\label{Good}
Let $X$ be a $C^1$ Anosov vector field on a $3$-manifold and consider a closed orbit $\gamma$ of the flow with orientable invariant manifolds. A sufficiently small tubolar neighborhood of this orbit is cut by the invariant weak manifolds in four different sectors. In each of this sectors there is a smoothly embedded compct annulus $C$ transverse to the flow and parallel to the invariant weak manifolds containing $\gamma$.
The closed orbit $\gamma$ can be enclosed in a neighborhood $N_C$ with smooth torus boundary such that $C\subset \partial N_C$ and the flow points outwards on $C$. We set $M_C=M\setminus \text{int}(N_C)$. Goodman first splits $M$ in two pieces $M_C$ and $N_C$ and gluing back them together using a diffeomorphism $\phi:\partial N_C\rightarrow \partial M_C$ with support in $C$. On the surgery annulus $C$ there is a natural coordinate system given by two sets of curves described by parameters $(l,m)$. The $l$-curves are closed and parallel to $\partial C$ while the $m$-curves have endpoints on the boundary. The diffeomorphism can be described by a {\it shear} on $C$.
$$F:C\rightarrow C,\;\;\; (l,m)\rightarrow (l+f(m),m),$$
where $$f:[-\epsilon,\epsilon]\rightarrow S^1,\;\;\;v\rightarrow f(m),$$
and $q\in \mathbb{Z}$ and $f:\mathbb{R}\rightarrow[0,2\pi]$ is a nondecreasing smooth function such that $f((-\infty,-\epsilon))=0$ and $f([\epsilon,\infty))=2\pi q$ is a non decreasing smooth function. This operation defines a $C^1$ smooth vector field $\tilde{X}$ in the new manifold (see \cite{Sha}).

Topologically, Goodman's surgery is a $(1,q)$-surgery and yields manifolds homeomorphic to the ones constructed in \fullref{3.1} removing a solid torus neighborhood $N$ of $C$ and gluing back a new solid torus $S^1 \times D^2$ sending the meridian $\mu_0$ of $\partial{(S^1 \times D^2)}$ to the curve $\mu+q\lambda$ of $M\setminus N$ where $\mu\subset N\cap C$.

\begin{figure}
\label{Goodmann}

\includegraphics[width=0.9\textwidth]{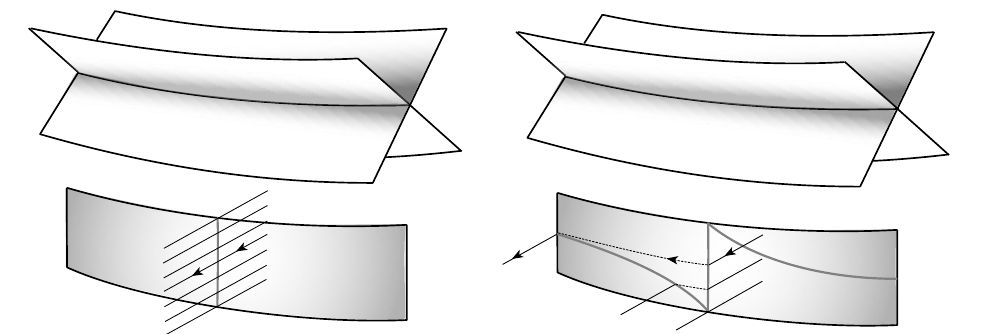}
  \caption{Goodman surgery near a closed orbit.}
\label{Goodmann}
\end{figure}

\subsubsection{Anosovity of Goodman's flows}
Anosovity is ensured for $q>0$ if we choose an annulus in the first and third quadrant, and for $q<0$ if the annulus is choosen in ther second or forth quadrant. The proof relies on the {\it cone criterion for hyperbolicity} and consists to show the existence of a cone field that is positively invariant and exponentially contracted along the flow. The existence of annuli with different preferred directions in a neighborhood of a closed orbit, implies that Goodaman surgery can be performed for every $q\in \mathbb{Z}
$.

\subsubsection{Foulon and Hasselblatt construction}

Given a $C^k$ contact form $\beta_+$, its Reeb vector field $R_{\beta_+}$ and a Legendrian knot $L$ for $\eta=\ker \beta_+$ Foulon and Hasselblatt describe in \cite{FoHa1} a family of contact surgeries of the \textit{Dehn} type that generates a new $C^k$ contact form $\tilde{\beta}_+$ in the new manifold $\tilde{M}$ and a $C^{k-1}$ vector field $R_{\tilde{\beta}_+}$. In a neighborhood $N$ of $L$ there is a coordinate system

$$(s,t,w)\in N= (-\delta,\delta)\times S^1 \times (-\epsilon,\epsilon),$$
where the parameters $(s,w)$ are defined on the \textit{surgery annulus} $C=\{0\}\times S^1 \times (-\epsilon,\epsilon)$ (see \fullref{fig:Anosov flow}). More precisely, $s \in S^1$ is the parameter of $L$ and $w$ belongs to some interval $(-\epsilon,\epsilon)$. The transverse parameter $t$ is such that the Reeb vector field of $\beta_+$ satisfies $R_{\beta_+}=\frac{\partial}{\partial t}$. In this coordinate system a contact form defining a contact structure takes the particularly simple expression $\gamma=dt+w\: ds$.

The surgery can be thought as first cutting the manifold $M$ along the annulus $C$ and then gluing back the two sides of the cut in a different way. In particular, we glue the point that on one side of the cut is described by coordinates $(s,w)$ to the one described by coordinates $(s+f(w),w)$ on the other side. Here $f:[-\epsilon,\epsilon]\rightarrow S^1$ is a non decreasing smooth function such that $f(-\epsilon)=0$ and $f(\epsilon)=2\pi q$. The shear map $F:C\rightarrow  C,\; (s,w)\rightarrow (s+f(w),w)$ is similar to the one defined in Goodman's construction and it produces new Reeb flows in $\tilde{M}$ for every $q\in \mathbb{Z}
$.

\label{AAA}
\subsubsection{Anosovity of the flows generated by Foulon and Hasselblatt construction}

Let $X$ be the generating vector field the geodesic flow $\phi^t$ on the unit tangent bundle of an hyperbolic surface $S$. Consider the knot $L$ obtained by rotating a closed geodesic along a fiber of an angle $\theta=\pi/2$. $L$ is Legendrian and a positive Foulon-Hasselblatt surgery yields a contact flow that is also Anosov. One important step of the proof of Anosovity of the flows generated by Foulon and Hasselblatt construction from a geodesic flow consists in showing that the cones constructed on one side of the surgery annulus are pushed by a positive shear map $F$ inside of the cones constructed on the other side of the surgery annulus (see \cite{FoHa1} and \cite{FoHa2}).

In a contact Anosov flow different from the geodesic  it is not clear if, after the application of the shear, the cones constructed on ones side of the surgery annulus will fit into the cones constructed on the other side, even when the Legendrian knot is transverse to the weak invariant foliations (see \fullref{EX}).

\section{Surgery on projectively Anosov flows}

\label{SPA}

Throughout the rest of the paper, we consider projectively Anosov flows and supporting bi--contact structures as equivalent objects. 

In \cite{FS1} we introduced a Dehn type surgery along a Legendrian-transverse knot $K$ in a $3$-manifold $M$ endowed with a bi--contact structure $(\xi_+,\xi_-)$, satisfying certain requirements near $K$. The procedure can be thought as first cutting the manifold $M$ along an annulus $A$ and then gluing back the two sides of the cut adding a $(1,q)$-Dehn twist. Contrarily to Goodman surgery this construction uses s surgery annulus that is tangent to the vector field $X$ defined by $(\xi_-,\xi_+)$. 

In this section we introduce a type of bi-contact surgery that uses a transverse annulus $C$ 
in a general projectively Anosov flow. The strategy for defining the surgery is similar to the one used in \cite{FS1}: we first find a suitable coordinate system where two contact forms $(\alpha_-,\alpha_+)$ defining the bi-contact structure take a simple expression. From now on, $\xi_-$ denotes the contact structure containing $TK$ while $\xi_+$ denotes the transverse one (see  \fullref{Transverse}).

\begin{lemma}
\label{coord}
There is a neighborhood $\Lambda=C\times (-\tau,\tau)$ of a Legendrian-transverse knot $K$, where $\tau\in \mathbb{R}^+$ and $C=S^1\times [-\epsilon,\epsilon]$, and coordinates $(w,s,t)$ in $\Lambda$ such that the $s$-curves with $s\in S^1$ and the $w$-curves with $w\in [-\epsilon,\epsilon]$, are given by the characteristic foliations induced on $C$ by the following pair of contact forms $(\alpha_-,\alpha_+)$ supporting $(\xi_-,\xi_+)$ 
$$\alpha_-=dw+a(t)\:ds$$
$$\alpha_+=ds+b(t)\:dw$$
with $\alpha_-=dw$, $\alpha_+=ds$ on $C$. Here $t\in (-\tau,\tau)$ is the parameter given by the flow of $X$. 

\end{lemma}

\begin{figure}
\label{Transverse}

\includegraphics[width=0.9\textwidth]{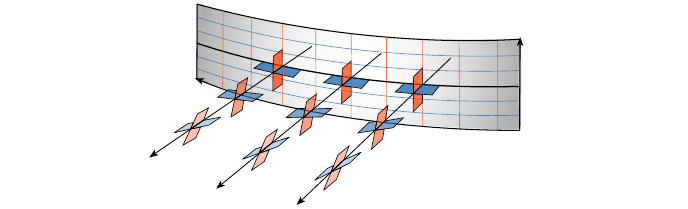}
  \caption{The surgery annulus $C$.}
\label{Transverse}
\end{figure}

\begin{proof}
Let $(\alpha_{1-},\alpha_{1+})$ be a pair of contact forms supporting $X$ and such that $TK\subset \ker \alpha_{1-}$. Consider the parameter $u$ defined by the flow $\phi^u$ of the Reeb vector field $R_{\alpha_{1-}}$ of $\alpha_{1-}$. The transverse surgery annulus $C$ is defined as follows. We start with the Legendrian-transverse knot $K:S^1\rightarrow M$ and we set $$C=\bigcup_{u\in[{-\nu,\nu}]}\phi^u(K)$$
for a sufficiently small $\nu\in \mathbb{R}^+$ so that $\phi^u(K)$ is transverse to $\xi_{1+}=\ker\alpha_{1+}$. The characteristic bi-foliation induced by the bi-contact structure $(\xi_{1-}=\ker \alpha_{1-},\xi_{1+}=\ker\alpha_{1+})$ is composed by two families of transverse curves on $C$. This induces a coordinate system $(w,s)$ where the $w$-curves, induced by the characteristic foliation of  $\xi_{1+}$ are transverse to $\partial C$, while the $s$-curves induced by  $\xi_{1-}$ are closed (since we used $R_{\alpha_{1-}}$ to span $C$) and parallel to $\partial C$. The transverse parameter $t$ is determined by the flow $\phi^t$ of $X$ and we set $t=0$ on $C$. Note that the $s$-curves are Legendrian for  $\xi_{1-}$ on $C$ while they are transverse to the contact planes of  $\xi_{1-}$ everywhere else in a neighborhood of $C$. The $w$-curves are instead  everywhere transverse to  $\xi_{1-}$ in a neighborhood of $C$, while the $t$-curves are Legendrian curves by definition. This implies that in a neighborhood of $C$ a contact form $\alpha_{2-}$ defining $\xi_{1-}$ can be written as
$$\alpha_{2-}=dw+a_2(w,s,t)\:ds,\;\;\;\alpha_{2-}=dw \;\;\;\text{on}\;\;\;C.$$
Similarly, the $s$-curves are everywhere transverse to $\xi_{1+}$ in a neighborhood of $C$, while the $w$-curves are Legendrian for $\xi_{1+}$ on $C$ and transverse everywhere else in a neighborhood of $C$, while the $t$-curves are Legendrian by definition. Therefore the contact structure $\xi_{1+}$ is defined by a contact form $\alpha_{2+}$ that can be locally written 
$$\alpha_{2+}=ds+b_2(w,s,t)\:dw,\;\;\;\alpha_{2+}=ds \;\;\;\text{on}\;\;\;C.$$
Rotating the contact planes of $\xi_{1+}$ and $\xi_{1-}$ independently and smoothly along the $t$-curves do not change the flow-lines. Therefore there is a pair of contact forms $(\alpha_-,\alpha_+)$ supporting $X$ that locally can be written as follows
$$\alpha_-=dw+a(t)\:ds,\;\;\;\alpha_+=ds+b(t)\:dw\;\;\;\text{with}\;\; t\in (-\tau,\tau),$$
with $a(0)=0$ and $b(0)=0$.
Note that since in the coordinate system $(s,t,w)$ the positive natural volume form is $dV=ds\wedge dt \wedge dw$, therefore $\alpha_-\wedge d\alpha_-=-a'\:dV$ while $\alpha_+\wedge d\alpha_+=b'\:dV$.

\end{proof}

\subsection{Definition of the shear and the perturbations}
\label{Deff}

\begin{figure}
\label{Perturbation}

\includegraphics[width=0.87\textwidth]{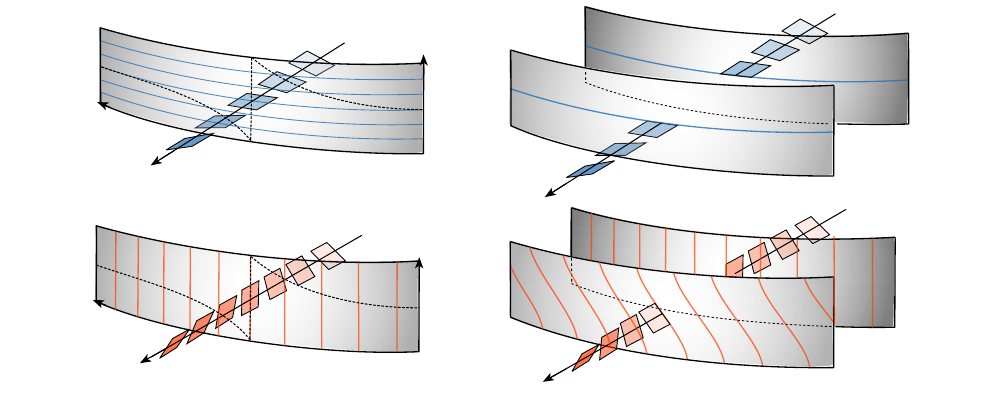}
  \caption{The shear $F$ does not preserve the contact forms. We introduce perturbations in order to produce a smooth plane field.}
\label{Perturbation}
\end{figure}

Without loss of generality we assume that $K$ is a Legendrian-transverse knot for the bi-contact structure $(\xi_-,\xi_+)$. Consider a coordinate system and an annulus $C$ as in \fullref{coord}. We define the shear along the surgery annulus $C$ similarly to \cite{FoHa1}, \cite{FoHa2} and \cite{FS1}. In particular $$f(w):[-\epsilon,\epsilon]\rightarrow S^1,\;\;\;v\rightarrow qg(w/\epsilon),$$
where $q\in \mathbb{Z}$  and $g:\mathbb{R}\rightarrow[0,2\pi]$ is a non-decreasing smooth function such that $g((-\infty,-1))=0$ and $g([1,\infty))=2\pi$ and $0\leq g'\leq4$. The shear $$F:C\rightarrow C,\;\;\; (s,w)\rightarrow (s+f(w),w),$$ identifies the point defined by the coordinates $(s,w)$ on the side of the annulus $C^-$ where the coordinate $t\leq 0$ to the point $(s+f(w),w)$ on the side of the annulus $C^+$ where $t\geq 0$. This operation  
defines a smooth $3$-manifold $\tilde{M}$. Since  we have  
$$F^*(\alpha_-)^+=dw+a(t)\:d(s+f(w))=\alpha+a(t)f'(w)\:dw,$$  $\alpha_-$ is not preserved by the shear. We perturb $\alpha_-$ to a 1-form $\tilde{\alpha}_-$ defined on $\tilde{M}\setminus C$ in such a way that the plane field distribution $\ker \tilde{\alpha}_-$ stays transverse to $\ker \alpha$ and defines a smooth contact structure after applying the shear. To this end we set $(\tilde{\alpha}_-)^-=(\alpha_-)^-$ and $(\tilde{\alpha}_-)^+=(\alpha_-)^+- \rho$ with 
$$\rho=\lambda_1(t)a(t)f'(w)\:dw,$$
where $\lambda_1:\mathbb{R}\rightarrow [0,1]$ is a smooth bump function with support in a neighborhood of $C$ defined by $t\in [0,\tau_1]$ and such that $\lambda(0)=1$. With this choice $\tilde{\alpha}_-$ is a 1-form defined on the whole new manifold $\tilde{M}$ as the following computation shows

$$F^*(\alpha_-)^+=dw+a(t)\:ds+a(t)f'(w)\:dw-a(t)f'(w)\:dw=(\alpha_-)^-.$$

Similarly, $\alpha_+$ is not preserved by the shear since on $C$ we have
$$F^*(\alpha_+)^-=d(s+f(w))+b(t)\:dw=ds+f'(w)\:dw.$$
We set $(\tilde{\alpha}_+)^-=(\alpha_+)^-$ and
$(\tilde{\alpha}_+)^+=(\alpha_+)^+- \sigma$ 
$$\sigma=\lambda_2(t)f'(w)\:dw$$
where $\lambda_2:\mathbb{R}\rightarrow [0,1]$ is a smooth bump function with support in $[0,\tau_2]$, such that $\lambda(0)=1$.
With this choice $\tilde{\alpha}_+$ is a 1-form defined on the whole new manifold since on $C$   $$F^*(\alpha_+-\sigma)^+=d(s+f(w))+(b(t)-\lambda_2(t)f'(w))\:dw=ds+b(t)\:dw=(\alpha_+)^-.$$

\begin{theorem}
\label{Bi}
Given a Legendrian-transverse knot in a bi-contact structure there is an infinite family of Dehn type surgeries that yields bi-contact structures in the new manifold.
\end{theorem}

\begin{proof}

The 1-form $\tilde{\alpha}_-$ for $t\geq 0$ is a contact form if and only if

$$\tilde{\alpha}\wedge d\tilde{\alpha}=\alpha \wedge d\alpha-\lambda'(t)a(t)^2 f'(w)\:dw \wedge ds\wedge dt\neq0.$$
If we deform the contact structure along the $t$-curves in such a way that $a(t)=t$ in a neighborhood of $C$ described by $t\in [0,\tau_1]$ the condition becomes
$$1-\lambda'(t)t^2 f'(w)=1-\lambda'(t)t^2 \frac{q}{\epsilon}\:g'(w/\epsilon)\neq 0.$$
where, similarly to \cite{FoHa1} we chose $\lambda_1:\mathbb{R}\rightarrow [0,1]$ smooth bump function with support in $[0,\tau_1]$ such that $\lvert \lambda'_1 \rvert\leq \pi/\tau_1$.
If we chose $$0<\tau_1<\frac{\epsilon}{\pi \left|q\right|},$$
we get 
$$\left| \lambda_1'(t)t^2 \frac{q}{\epsilon}\:g'(w/\epsilon)\right|\leq\left|\frac{\pi \tau_1q}{\epsilon}\right|<1,$$
therefore $\tilde{\alpha}\wedge d\tilde{\alpha}\neq0.$ Note that for every surgery coefficient $q\in \mathbb{Z}
$ the surgery produce a new contact 1-form.

The contact condition on 
$\tilde{\alpha}_+$ translates in the following inequality $$b'(t)-\lambda_2'(t) f'(w)=b'(t)-\lambda_2'(t) \:\frac{q}{\epsilon}\:g'(w/\epsilon)\neq0.$$
Since $\lambda_2'(t)\leq 0$, the contact condition is verified if $q\in \mathbb{N}$.

\end{proof}

\section{Dynamical  interpretation of the surgery} 

\label{Dy}

Hozoori shows in \cite{Hoz2} that if $\phi^t$ is a $C^1$-Hölder  Anosov flow with orientable weak invariant foliations, there is a special pair of supporting contact forms $(\alpha_-,\alpha_+)$ such that, in a neighborhood of a periodic orbit $\gamma$, the Reeb vector field $R_{\alpha_-}$ of $\alpha_-$ belongs to $\ker \alpha_+$. This implies that a closed orbit $\gamma$ can be pushed to a Legendrian-transverse knot $K=\phi^{\overline{w}}(\gamma)$ (were $\overline{w}>0$ small) using the flow $\phi^{w}$ of the vector field $R_{\alpha_-}$. Indeed, since $R_{\alpha_-}$ is Legendrian for $\ker \alpha_+$, the contact plane rotate all in the same direction along the flowlines of $R_{\alpha_-}$, while the planes of $\ker \alpha_-$ stay invariant.

\begin{corollary}[Hozoori \cite{Hoz2}]
Let $\phi^t$ be a $C^1$-Hölder Anosov flow on a $3$-manifold with orientable weak invariant foliations. There is a supporting bi-contact structure such a neighborhood of a periodic orbit $\gamma$ contains an isotopic Legendrian-transverse knot $K$.
\end{corollary}

We call a knot $K$ obtained by flowing a closed orbit $\gamma$ to a Legendrian-transverse knot a {\it Legendrian-transverse push-off of $\gamma$}.
The above result allows us to construct the transverse annulus $C$ and to apply the surgery as in lemma \fullref{coord} and \fullref{Deff}. We now show that, when the flow is Anosov, there is a strong relation between our bi-contact surgery and Goodman's operation.

\begin{corollary}
Suppose that $\phi^t$ is an Anosov flow with orientable invariant foliations. Consider the Legendrian-transverse push-off $K$ of a periodic orbit $\gamma$. Let $C$ be a transverse annulus as in \fullref{coord} and let $F:C\rightarrow C$ be a shear as in \fullref{Deff}. With these choices, Goodman surgery generates the same flows (up to reparametrization) obtained by the bi-contact surgery of \fullref{Deff}.
\end{corollary}

\begin{proof}
The perturbation of the bi-contact structure defined in \fullref{Deff} does not change the flow lines. Therefore for the same choice of the shear and the surgery annulus, the bi-contact surgery reconstructs the same flow lines of Goodman's construction and vice versa (see \fullref{Goodmann} and \fullref{Transverse}).
\end{proof}

\section{Construction of  Lyapunov--Lorentz metrics using bi-contact geometry}

\label{Anosovity}

In this section we use bi-contact geometry to construct Lorentzian metrics that satisfy Barbot's version of the cone criterion for hyperbolicity.

\begin{theorem}[\cite{Bar1}]
\label{prop:Barbot}
A $C^1$ flow $\phi^t:M\rightarrow M$ of a $3$-manifold is an Anosov flow if and only if there are two continuous Lorentz metrics $Q^+$ and $Q^-$ on $M$ and constants  $a,b,c,T>0$ such that
\end{theorem}

\begin{enumerate}
\item {\it for all} $w\in T_xM, t>T$, if $Q^\pm(w)>0$ {\it then} $Q^\pm(D_x\phi^{\pm t}(w))>ae^{ bt}Q^\pm(w),$

\label{cond:b1}

\item $\mathcal{C}^+\cap \mathcal{C}^-=\varnothing$ {\it where} $\mathcal{C}^\pm$ {\it is the} $Q^{\pm}$-{\it positive cone},

\label{cond:b2}

\item $Q^\pm(X)=-c$ {\it where} $X$ {\it is the generating vector field},

\label{cond:b3}

\item $D_x\phi^{\pm t}(\overline{\mathcal{C}^\pm(x)})\setminus\{0\}\subset \mathcal{C}^\pm(\phi^{\pm t}(x))$ {\it for all} $t>0$ {\it and all} $x.$

\label{cond:b4}
\end{enumerate}

Lorentzian metrics as above are called Lyapunov. In order to explicitly construct a pair of Lyapunov--Lorentz metrics induced by the bi-contact structure we introduce the following results by Hozoori.
\begin{proposition}[\cite{Hoz}]
\label{p2}
Let $X$ be an Anosov flow on a 3-manifold $M$ and let $r_s$ and $r_u$ be its expansion rates. There is a metric on $TM$ such that if $\eta=E^{ss}\oplus E^{uu}$ and $e_s\in E^{ss}$, $e_u\in E^{ss}$ are unit vectors, we have
$$D_x\phi^t(e_s)=e^{\int_0^tr_s(\tau)\:d\tau}e_s,$$
$$D_x \phi^t(e_u)=e^{\int_0^tr_u(\tau)\:d\tau}e_u.$$
\end{proposition}
We recall that the expansions rates as described above satisfy the inequality $r_s<0<r_u$. 

\begin{proposition}[\cite{Hoz}]
\label{p1}
In the hypotesis of \fullref{p2} there are $C^1$ 1-forms such that
 $$ \alpha_s(e_s)=1,\;\;\;\ker \alpha_s=E^u \oplus \langle X\rangle, $$
$$ \alpha_u(e_u)=1,\;\;\;\ker \alpha_u=E^s \oplus \langle X\rangle. $$
Moreover, the $C^1$ 1-forms $$\alpha_+=\alpha_u-\alpha_s,$$ $$\alpha_-=\alpha_u+\alpha_s,$$
are contact.
\end{proposition}

We use \fullref{p2} and \fullref{p1} to prove the following.

\begin{theorem}
\label{LL}
Let $\phi^t$ be an Anosov flow (transitive or not) with orientable weak foliations and generating vector field $X$. There is a pair of contact forms $(\alpha_-,\alpha_+)$ defining $X$ such that the pair of Lorentzian metrics 
$$Q^{\pm}=\pm \alpha_-\otimes \alpha_+-\beta^2$$
is Lyapunov.  
Here $\beta$ is the canonical invariant $1$-form such that $\ker\beta=E^{uu}\oplus E^{ss}$ and $ \beta(X)=1$.

\end{theorem}

\begin{proof}
Conditions \ref{cond:b2}, \ref{cond:b3} and \ref{cond:b4} easily follow from the definition of $Q^{\pm}$.
We show that the contact forms defined in \fullref{p1} satisfy condition \ref{cond:b1}.

Set $w=AX+Be_s+Ce_u$ where $e_s$ and $e_u$ are unit vectors respectively in $E^{ss}$ and $E^{uu}$. Since we have 
$$Q^{+}(w)=C^2-B^2-A^2,$$
and by \fullref{p2}  $$Q^+(D_x\phi^t(w))=C^2e^{2\int_0^tr_u(\tau)\:d\tau }-B^2e^{2\int_0^tr_s(\tau)\:d\tau}-A^2$$
$$\geq C^2e^{2\int_0^tr_u(\tau)\:d\tau }-B^2e^{2\int_0^tr_u(\tau)\:d\tau}-A^2e^{2\int_0^tr_u(\tau)\:d\tau }= e^{2\int_0^tr_u(\tau)\:d\tau }Q^{+}(w).$$ 

 For $Q^{-}$ we have 

$$Q^{-}(w)=-C^2+B^2-A^2,$$
by \fullref{p2} we have $$Q^-(D_x\phi^t(w))=-C^2e^{2\int_0^{-t}r_u(\tau)\:d\tau }+B^2e^{2\int_0^{-t}r_s(\tau)\:d\tau}-A^2$$
$$\geq -C^2e^{2\int_0^{-t}r_s(\tau)\:d\tau }+B^2e^{2\int_0^{-t}r_s(\tau)\:d\tau}-A^2e^{2\int_0^{-t}r_s(\tau)\:d\tau }= e^{2\int_0^{-t}r_s(\tau)\:d\tau }Q^{-}(w)=e^{2\int_0^{t}-r_s(\tau)\:d\tau }Q^{-}(w).$$ 

If we set $k=\min\{\inf(\{r_u(t)\}_{t>0}),\inf(\{-r_s(t)\}_{t>0})\}$, we have

$$Q^{\pm}(D_x\phi^t(w))\geq e^{2\int_0^t k\:d\tau }Q^{\pm}(w)\geq e^{bt }Q^{\pm}(w).$$

\end{proof}

\begin{figure}

\includegraphics[width=0.9\textwidth]{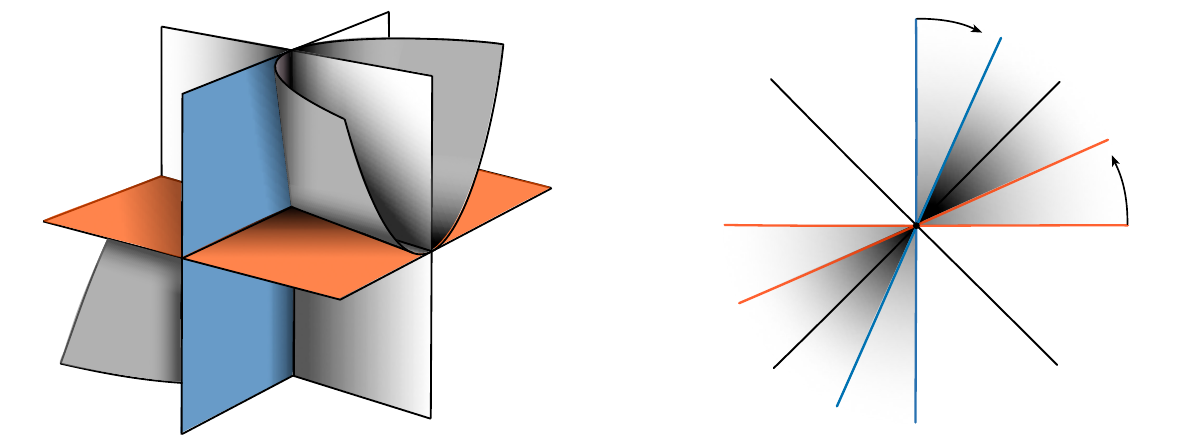}
  \caption{The bi-contact structure is represented by the red and blue planes. The white plane is the preserved $C^0$ plane field $\ker \beta $.}
\label{fig:Coordinates1}
\end{figure}

Hozoori shows in \cite{Hoz3} that there is a reparametrization of a volume preserving Anosov flow $X$ called Liouville reparametrization and denoted by $X_L$ such that for a pair of contact forms $\alpha_-$ and $\alpha_+$ the  $C^0$ plane field defined by the $\langle R_{\alpha_-}, R_{\alpha_+}\rangle$ is preserved by $X_L$. Therefore we have the following.
\begin{corollary}
\label{LL2}
If an Anosov vector field $X$ is volume preserving there is a reparametrization $X_L$ and a pair of contact forms $(\alpha_-,\alpha_+)$ supporting $X_L$ with Reeb vector fields $(R_{\alpha_-},R_{\alpha_+})$ such that the pair of Lorentzian metrics 
$$Q_{X_L}^{\pm}=\pm \alpha_-\otimes \alpha_+-\alpha_{X_L}^2$$
is Lyapunov.  
Here $\alpha_{X_L}$ is the  $1$-form such that $\ker\alpha_{X_L}=\langle R_{\alpha_-} ,R_{\alpha_+} \rangle$ and $ \alpha_{X_L}(X_L)=1$.
\end{corollary}

\begin{remark}
\label{EX}
Consider a general contact Anosov flow and let be $(\alpha_-,\alpha_+,\beta_+)$ as in \fullref{LL}. Consider a Legendrian knot $L$ for the preserved contact form $\beta_+$ and suppose that $L$ is also transverse to the weak invariant foliations. By Foulon and Hasselblatt construction there is a coordinate system $(w,s,t)$ and a transverse annulus $C$ such that $\beta_+=dt+w\:ds.$
The contact forms defining the flow take the expressions
$$\alpha_-=a_1\:dw+a_2\:ds,$$
$$\alpha_+=b_1\:ds+b_2\:dw.$$

Since the knot is transverse to the weak invariant foliations we assume without loss of generality that the $w$-curves coincide with the leaves of the characteristic foliation induced on $C$ by $\xi_+$. Therefore we have $b_1(w,s,t=0)$ and the metrics of \fullref{LL} on $C$ can be expressed as

$$Q^{\pm}=\pm (k_1\:dw\:ds+k_2\:ds^2)-(dt+w\:ds)^2$$ 
 
where $k_1=b_1(w,s,t=0) a_1(w,s,t=0)$ and $k_2=b_1(w,s,t=0) a_2(w,s,t=0)$.
As mentioned in \fullref{AAA} it is not in general clear if the shear reconstruct a strictly contracting cone field. An example of the issues we might encounter is the following. Consider the generic vector $v=W\:\frac{\partial}{\partial w}+\:S\frac{\partial}{\partial w}$ on $TC$. The traces of these cones on $C$ along the knot $L$ (therefore $w=0$) satisfies the equation
$$s\:(k_1\:w+k_2\:s)=0.$$
This corresponds in the plane $sw$ to the vertical axis and the line $l$ given by the equation $k_1\:w+k_2\:s=0.$
A positive shear rotates $l$ clockwise if the slope of $l$ is positive and counterclockwise if the slope of $l$ is negative. If the slope of $l$ is positive and the surgery coefficient is large enough, it is not possible to reconstruct a new bi-contact structure in the new manifold, therefore there is not a new strictly contracting cone field satisfying the cone criterion of hyperbolicity.

In the case of the geodesic flow the metrics on $C$ take the expression
$$Q^{\pm}=\pm k_1(w)\:dw\: ds-(dt+w\:ds)^2.$$ Note that along the Legendrian knot $L$ the metrics reduce to the ones used by Foulon and Hasselblatt in \cite{FoHa1}. In particular the traces of the cones on $C$ show up as quadrants ($sw=0$). If the shear is positive it pushes the axis $s=0$ inside the cones while preserves the axis $w=0$. Therefore the new cone field is strictly contracting. Note that the trace of the cones show up as quadrants if $L$ is a Legendrian knot for both $\xi_+$ and $\eta_+$. In the next section we show that being Legendrian for both $\xi_-$ and $\eta_+$ is a sufficient condition that allows us to produce new contact Anosov flows by surgery along $L$. 
\end{remark}

\begin{figure}

\includegraphics[width=0.9\textwidth]{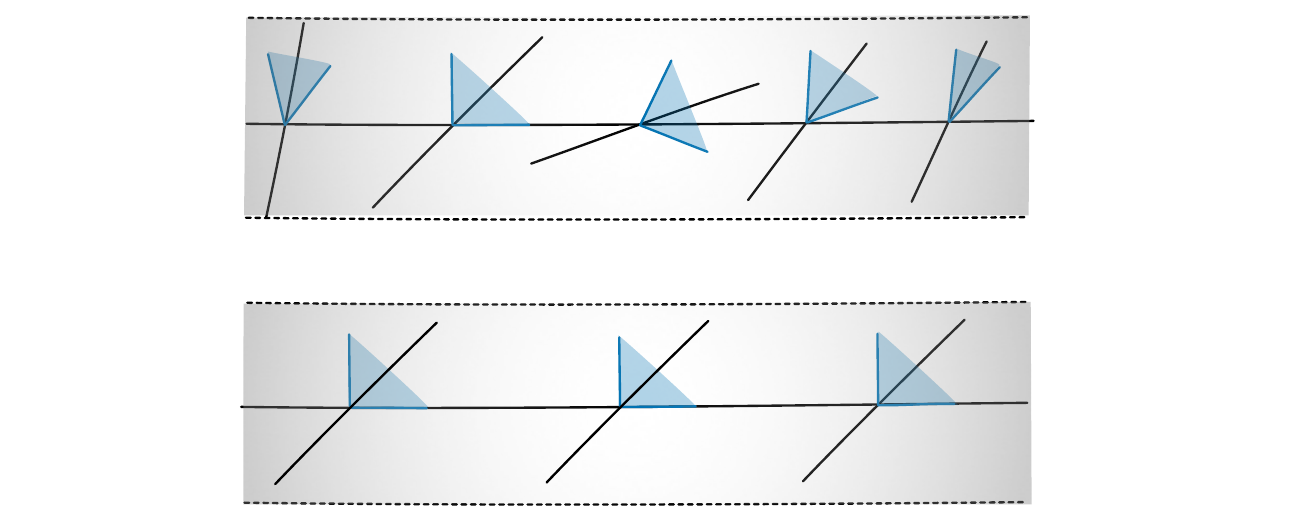}
  \caption{On the top: the surgery annulus centered on a knot transverse to the weak invariant foliations in a general contact Anosov flow. In blue the traces of the cones constructed around one of the weak invariant foliations (in black). On the bottom: the situation on the geodesic flow. The cones show up as quadrants.}

\end{figure}

\section{New contact Anosov flows from any contact Anosov flow} 

In this section we show that in every contact Anosov flow there are families of Legendrian knots that can be used to generate new contact Anosov flows by surgery and we study their abundance. We start to show the following.

\begin{proposition}
\label{thm:71}
Let $\phi^t$ be any contact Anosov flow defined by the bi-contact structure $(\xi_-,\xi_+)$ and preserving a positive contact structure $\eta_+$. Consider a knot $L$ that is Legendrian for the pair of contact structures $(\eta_+,\xi_-)$ (or for the pair $(\eta_+,\xi_+)$). The knot $L$ is Legendrian-transverse for $(\xi_-,\xi_+)$.
\end{proposition}

\begin{proof} 

Since $TL\in \xi_-$ (or $TL\in \xi_+$) by assumption, we are left to show that $L$ is transverse to $\xi_+$ (or $\xi_-$). 
Suppose $L$ is not transverse to $\xi_+$ (or $\xi_-$), there is a $p\in L$ such that $T_pL\subset \xi_+$ (or $T_pL\subset \xi_-$). Therefore $T_pL\subset (\xi_-\cap \xi_+)_p$ and since $( \xi_-,\xi_+)$ supports a Reeb flow of $\eta_+$ we have that $T_pL$ is transverse to $\eta_+$ in $p$ contradicting the fact that $L$ is Legendrian for $\eta_+$. 
\end{proof}
 
\begin{remark}
Consider the geodesic flow on the unit tangent bundle of an hyperbolic surface with the orientation that makes it preserving a positive contact form. The Legendrian knots obtained by rotating a closed orbit of an angle of $\pi/2$ are orbits of the flow defined by the pair $(\eta_+,\xi_-)$. It is well known that the flow defined by $(\eta_+,\xi_-)$ is Anosov and homotopic to the geodesic flow. On the other hand, the fibers of the Seifert fibration are all the trajectories of the flow defined by the pair of (positive) contact forms $(\eta_+,\xi_+)$. 
\end{remark}
Since a Legendrian knot as in \fullref{thm:71} is also transverse to the weak invariant foliations the question of their abundance is strictly related to a similar one asked by Foulon and Hasselblatt (\cite{FoHa1} section 5.3) that can be paraphrased as follows. 
\begin{question}
\label{Q}
Given a contact Anosov flow preserving a contact structure $\eta$ and a knot $K$, is there a Legendrian knot $L$ for $\eta$ that is transverse to the weak invariant foliations and that is isotopic to the given knot $K$?  
\end{question} 

It is known that a knot $K$ in a contact manifold is $C^0$-close to a Legendrian knot. The property of being transverse to the weak invariant foliations introduce some topological restrictions to the isotopy class of $K$ (e.g. $K$ cannot be topologically trivial since it is transverse to a taut foliation).

To study the abundance of the knots of \fullref{thm:71} we first notice that they can be interpreted as the periodic orbits of a pA flow $\psi^t$ defined by the bi-contact structure $(\xi_-,\eta_+)$ (or $(\xi_+,\eta_+)$).

\begin{proposition}[Proposition 1.1 in the introduction]
\label{74}
Let $(\xi_-,\xi_+)$ a bi-contact structure defining a contact Anosov flow $\phi^t$ that preserves a contact structure $\eta_+$. The $\psi^t$ defined by the bi-contact structure $(\xi_-,\eta_+)$ is Anosov.  
\end{proposition}

\begin{proof} 
 Since $X$ is Reeb for $\eta_+$ and $X\subset \xi_-$ by Hozoori's criterion the bi-contact structure $(\xi_-,\eta+)$ is Anosov.

\end{proof}

A consequence of \fullref{74} is that in any contact Anosov flow there are infinitely many Legendrian knots transverse to the weak invariant foliations.

 We finally prove the main result of this section.

\begin{theorem}
\label{thm:7}
Let $\phi^t$ be any contact Anosov flow defined by the bi-contact structure $(\xi_-,\xi_+)$ and preserving a positive contact structure $\eta_+$. Consider a knot $L$ Legendrian for the pair of contact structures $(\eta_+,\xi_-)$ (or for $(\eta_+,\xi_+)$). A $(1,q)$-bi-contact surgery along $L$ yields a contact Anosov flow for every $q\in \mathbb{N}$.
\end{theorem}

\begin{proof}
We prove the statement for a knot $L$ Legendrian for the pair of contact structures $(\eta_+,\xi_-)$. Since $L$ is Legendrian for $\eta_+$ and transverse to $\xi_+$ we can construct a smooth vector field directed by $(\eta_+,\xi_+)$ and use it to flow $L$ to generate an annulus $C$. On $C$ there is a coordinate coordinate system $(s,w)$ where $s$ is the coordinate describing the knot and the $w$-curves are directed by $(\eta_+,\xi_+)$. Moreover, the parameter $w$ can be chosen such that $\beta_+=dt+w\:ds$ (see \cite{FoHa1}), where $t$ is the parameter of the Anosov flow $X$. Since $L$ is a Legendrian knot for $\xi_-$ we can use the flow of $X$ to construct a flow box bounded by $\phi^{-\tau}(C)=C_{in}$ and $\phi^{\tau}(C)=C_{out}$ for $\tau$ sufficiently small. If $C$ is sufficiently thin, the caracteristic foliation induced by $\xi_-$ on $C_{in}$ has positive slope, while the one induced on $C_{out}$ has negative slope. Therefore there is a pair of supporting contact forms that in a flow box neighborhood of  $ C$ can be expressed as $\alpha_-=dw+a(t)\:ds$ and $\alpha_+=ds+b(t)\:dw$
with $a(0)=0,\:b(0)=0$. This implies that $L$ is Legendrian for $\ker \alpha_-$ and $\ker \beta_+$ and we can perform a $(1,q)$-Legendrian-transverse  surgery on $(\alpha_-,\alpha_+)$ using the same shear of Foulon and Hasselblatt construction. For positive surgery coefficients, the resulting vector field is contact and projectively Anosov and therefore Anosov by \cite{Hoz3}.

\end{proof}

\section{Applications to surgery on contact structures} 

\label{contact}

The formulation of Goodman surgery in terms of bi-contact structures clearly shows the relationships between the contact structures supporting the initial vector field $X$ and the ones defining the new vector field $\tilde{X}$. In the following we summarize these relations. We remark that the same statements hold for the bi-contact surgery introduced in \cite{FS1}.

\begin{proposition}
Suppose that we have a Legendrian-transverse knot $K$ in a bi-contact structure $(\xi_-,\xi_+)$. If $q>0$, the $(1,q)$-surgery of \fullref{Deff} acts as a {\it Legendrian surgery} on $\xi_-$ while it acts as a inadmissible transverse surgery on $\xi_+$.

Viceversa, suppose that $K$ is Legendrian for $\xi_+$ and transverse to $\xi_-$. 
 If $q<0$ a $(1,q)$-surgery acts as a {\it Legendrian surgery} on $\xi_+$ while acts as an inadmissible transverse surgery on $\xi_-$.

The surgery coefficient $q$ is defined with respect to the contact framing induced by Legendrian contact structure.
\end{proposition}

\begin{proof}
Suppose that $K$ is Legendrian-transverse for $(\xi_-,\xi_+)$. There are neighborhood $N_L\subset N$ and $N_T\subset N$ contactomorphic respectively to a standard neighborhood of a Legendrian knot and to a standard neighborhood of a transverse knot (see \fullref{Dehn}). Perform the $(1,q)$-bi-contact surgery along $K$ applying the deformation of the contact structures in a neighborhood $N'\subset N_L\cap N_T$. After surgery consider the solid tori $\tilde{N}_L$ and $\tilde{N}_T$. These are neighborhoods of a Legendrian knot and a transverse knot respectively. Moreover $\partial{\tilde{N}_L}=\partial N_L$ and $\partial{\tilde{N}_T}=\partial N_T$. Consider on both $\partial{N_L}$ and $\partial{N_T}$ a preferred longitude $\lambda$ given by the contact framing of $\xi_-$ (the contact structure tangent to the knot). Since after a $(1,q)$-surgery along the annulus $C$ a  $\mu+q\lambda$ curve is contractible, the surgery of \fullref{Deff} acts as a $(1,q)$-Legendrian surgery on $\xi_-$ and as a $(1,q)$-transverse surgery on $\xi_+$. The surgery coefficient $q$ is defined with respect to the contact framing induced by $\xi_-$.

\end{proof}

Since the supporting bi-contact structures of an Anosov flow are tight (hypertight as shown by Hozoori in \cite{Hoz}), have the following (hyper)tightness results for classic contact and transverse surgeries. As usual, let $q$ be an integer.

\begin{corollary}
\label{Cor3}
Let $K$ be a Legendrian-transverse knot in a bi-contact structure $(\xi_-,\xi_+)$ defining an Anosov flow. Every inadmissible transverse $(1,q)$-surgery on $\xi_+$ along $K$ yields a hypertight contact structure $\tilde{\xi}_+$.
\end{corollary}

\begin{proof}
Every inadmissible transverse $(1,q)$-surgery on $\xi_+$ along $K$ can be realized by a $(1,q)$-bi-contact surgery defined in \cite{FS1} with coefficient $q>0$. This yields a bi-contact structure and therefore an Anosov flow (see \cite{FS1}), thus $\tilde{\xi}_+$ is hypertight. 
\end{proof}

The following result of Conway allows us to translate  \fullref{Cor3} in terms of positive contact surgeries.
\begin{theorem}[\cite{Con1}]
\label{Con}
Every inadmissible transverse surgery on a transverse knot corresponds to a positive contact surgery on a Legendrian approximation.
\end{theorem}   

Suppose that $\gamma$ is a periodic orbit of an Anosov flow and consider a supporting bi--contact structure $(\xi_-=\ker \alpha_-,\xi_+=\ker \alpha_+)$ such that $R_{\alpha_-}\in \ker \alpha_+$. Consider the transverse push--off of $\gamma$ along $R_{\alpha_-}$. Such a Legendrian-transverse knot can be pushed to the original closed orbit by a Legendrian push-off using the flow of $-R_{\alpha_-}$. The closed orbit is the Legendrian approximation of \fullref{Con} and we have the following.

\begin{corollary}

Let $\gamma$ be a closed orbit in an Anosov flow defined by a bi-contact structure $(\xi_-,\xi_+)$. A positive $(1,q)$-contact surgery along $\gamma$ on $\xi_+$ yields a new hypertight contact structure.
\end{corollary}

\begin{proof}
Push the closed orbit $\gamma$ to a Legendrian-transverse knot using the flow of $R_{\alpha_-}$. If $q>0$ the bi-contact surgery induces an inadmissible transverse surgery along $K$ on $\xi_+$ by construction. \fullref{Cor3} implies that the new contact structure $\tilde{\xi}_+$ is hypertight and \fullref{Con} ensures that the same contact structure $\tilde{\xi}_+$ can be realized (up to isotopy) as a contact $(1,q)$-surgerey along the Legendrian approximation of $K$ given by $\gamma$. 

\end{proof}


\section{Projectively Anosov flows on hyperbolic 3-manifolds }

\label{nonA}

We finally apply our construction to produce families of bi-contact structures that do not define Anosov flows. Our examples are generated applying the bi-contact surgery on suitable {\it propellers} as described by Mitsumatsu (\cite{Mit}, \cite{Mit2}, see \fullref{Bi}). In particular, we show how to generate new non-Anosov bi-contact structures on the family of $3$-manifolds $\{M(q),\; q\in \mathbb{N}\}$ as described in the introduction. We remark that $M(q)$ is hyperbolic except when $q\in \{0,1,2,3,4\}$.

Consider $M(0)$ equipped with a pair of vertically rotating tight contact structures $(\xi_-=\ker \alpha_-,\: \xi_n=\ker \alpha_{n})$ defined as follows. The negative one $\xi_-$ has minimal twisting (equivalently is a negative contact structure defining the suspension flow as defined in \fullref{example1}). The positive one $\xi_n$ is a vertical rotating positive contact structure with total  rotation $\theta_{n}>2\pi$ in $[0,1]$. Since $\theta_{n}>2\pi$ the contact structures are tangent along a finite number of horizontal tori. Let $\Sigma$ one of this tori. After an isotopy of both the defining contact structures we assume that the (common) characteristic foliation induced by $\xi_-$ and $\xi_n$ on $\Sigma$ is composed by closed leaves. The pair $(\xi_-,\xi_n)$ is not a bi-contact structure yet since the plane fields $\xi_-,\xi_n$ are not transverse. We introduce a deformation on $\xi_n$ and define $\alpha_+=\alpha_n+\epsilon\:dz$. The pair $(\alpha_-,\alpha_+)$ define a bi-contact structure.

\begin{corollary}
Consider a torus bundle over the circle equipped with the bi-contact structure $(\xi_-,\xi_+)$ described above. A suspension knot $K$ is a Legendrian-transverse knot and a $(1,q)$-Legendrian-transverse along $K$ yields a bi-contact structure $(\tilde{\xi}_-,\tilde{\xi}_+)$.
\end{corollary}

\begin{proof}
By construction $K$ is Legendrian for both $\xi_-$ and $\xi_n$. The contact  structure defined by  $\alpha_+=\alpha_n+\epsilon \:dz$ is transverse to $K$. Therefore if $q\in \mathbb{N}$ a $(1,q)$-Legendrian-transverse surgery yields a bi-contact structure by \fullref{Bi}. 
e new manifold is $M(q)$.
\end{proof}
Note that if the torus bundle is the manifold $M(0)$ described in the introduction and we consider a Legendrian-transverse knot $K$ the suspension of $(0,0)$, a $(1,q)$-Legendrian transverse surgery yields the manifold $M(q)$. As a consequence we have the following.

\begin{corollary}
The manifold $M(q)$ obtained by Goodman surgery along the knot $K$ obtained by suspension of $(0,0)$ in $M(0)$ supports infinitely many non-homotopic bi-contact structures that do not define Anosov flows. The associated pA flows have an invariant subsurface of genus $g>0$.

\end{corollary}

\begin{figure}
\label{NonAnosov}

\includegraphics[width=0.8\textwidth]{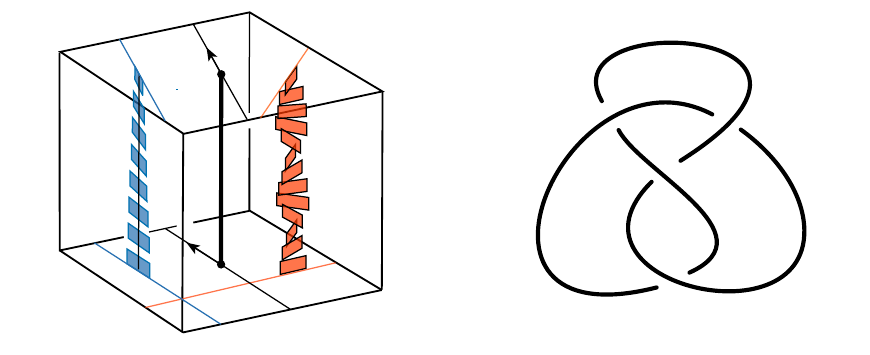}
  \caption{The negative contact structure $\ker \beta_n$ (in red) has non-minimal twisting. The positive one $\ker \alpha$ (in blue) defines the suspension flow. The framing of the surgery annulus associated to a suspension knot $K$ is determined by $\ker \alpha$. Clearly $K$ is Legendrian transverse for the perturbed plane field defined by the pair of contact forms $(\alpha,\beta_n+\epsilon\:dz)$. }
\label{NonAnosov}
\end{figure}

\begin{proof}
Consider $M(0)$ with the bi-contact structure described in the introduction of \fullref{nonA}. Since our operation perturbs the bi-contact structure just in a neighborhood $\Lambda$ of the Legendrian-transverse knot $K$, and $K$ is transverse to the invariant torus $\Sigma$ foliated by closed orbits, there is an invariant subsurface has genus $g>0$ in the new manifold containing an annulus foliated by closed orbits. Therefore, the new flow is pA and not Anosov. 
\end{proof}

%% file: Projectively_Goodman.bbl
\begin{thebibliography}{Hoz21b}

\bibitem[Ano63]{Ano1}
D.~V. Anosov.
\newblock Ergodic properties of geodesic flows on closed {R}iemannian manifolds
  of negative curvature.
\newblock {\em Dokl. Akad. Nauk SSSR}, 151:1250--1252, 1963.

\bibitem[Ano67]{Ano2}
D.~V. Anosov.
\newblock Geodesic flows on closed {R}iemannian manifolds of negative
  curvature.
\newblock {\em Trudy Mat. Inst. Steklov.}, 90:209, 1967.

\bibitem[ARH03]{AF}
Aubin Arroyo and Federico Rodriguez~Hertz.
\newblock Homoclinic bifurcations and uniform hyperbolicity for
  three-dimensional flows.
\newblock {\em Ann. Inst. H. Poincar\'{e} Anal. Non Lin\'{e}aire},
  20(5):805--841, 2003.

\bibitem[Asa04]{Asa}
Masayuki Asaoka.
\newblock A classification of three dimensional regular projectively {A}nosov
  flows.
\newblock {\em Proc. Japan Acad. Ser. A Math. Sci.}, 80(10):194--197 (2005),
  2004.

\bibitem[Bar06]{Bar1}
Thierry Barbot.
\newblock De l'hyperbolique au globalment hyperbolique.
\newblock Preprint, 2006.

\bibitem[BBP]{BBP}
Jonathan Bowden, Christian Bonatti, and Rafael Potrie.
\newblock Some remarks on projective {A}nosov flows in hyperbolic 3-manifolds.

\bibitem[BBY17]{BBY}
Fran\c{c}ois B\'{e}guin, Christian Bonatti, and Bin Yu.
\newblock Building {A}nosov flows on 3-manifolds.
\newblock {\em Geom. Topol.}, 21(3):1837--1930, 2017.

\bibitem[BE13]{BaEt1}
John~A. Baldwin and John~B. Etnyre.
\newblock Admissible transverse surgery does not preserve tightness.
\newblock {\em Math. Ann.}, 357(2):441--468, 2013.

\bibitem[Bla19]{BL}
David~E. Blair.
\newblock A survey of {R}iemannian contact geometry.
\newblock {\em Complex Manifolds}, 6(1):31--64, 2019.

\bibitem[BP98]{BLP}
David~E. Blair and D.~Peronne.
\newblock Conformally {A}nosov flows in contact metric geometry.
\newblock {\em Balkan J. Geom. Appl.}, 3(2):33--46, 1998.

\bibitem[Con19]{Con1}
J.~Conway.
\newblock Tight contact structures via admissible transverse surgery.
\newblock {\em J. Knot Theory Ramifications}, 28(4):1950032, 24, 2019.

\bibitem[DG01]{DiGe2}
Fan Ding and Hansj\"{o}rg Geiges.
\newblock Symplectic fillability of tight contact structures on torus bundles.
\newblock {\em Algebr. Geom. Topol.}, 1:153--172, 2001.

\bibitem[EG02]{EtGh}
John Etnyre and Robert Ghrist.
\newblock Tight contact structures and {A}nosov flows.
\newblock In {\em Proceedings of the 1999 {G}eorgia {T}opology {C}onference
  ({A}thens, {GA})}, volume 124, pages 211--219, 2002.

\bibitem[Eli90]{Elia}
Yakov Eliashberg.
\newblock Topological characterization of {S}tein manifolds of dimension
  {$>2$}.
\newblock {\em Internat. J. Math.}, 1(1):29--46, 1990.

\bibitem[ET98]{ElTh}
Yakov~M. Eliashberg and William~P. Thurston.
\newblock {\em Confoliations}, volume~13 of {\em University Lecture Series}.
\newblock American Mathematical Society, Providence, RI, 1998.

\bibitem[FH13]{FoHa1}
Patrick Foulon and Boris Hasselblatt.
\newblock Contact {A}nosov flows on hyperbolic 3-manifolds.
\newblock {\em Geom. Topol.}, 17(2):1225--1252, 2013.

\bibitem[FHV19]{FoHa2}
Patrick Foulon, Boris Hasselblatt, and Anne Vougon.
\newblock Orbit growth of contact structures after surgery.
\newblock Preprint, 2019.

\bibitem[Gay02]{Gay}
David~T. Gay.
\newblock Symplectic 2-handles and transverse links.
\newblock {\em Trans. Amer. Math. Soc.}, 354(3):1027--1047, 2002.

\bibitem[Gei08]{Ge2}
Hansj\"{o}rg Geiges.
\newblock {\em An introduction to contact topology}, volume 109 of {\em
  Cambridge Studies in Advanced Mathematics}.
\newblock Cambridge University Press, Cambridge, 2008.

\bibitem[Goo83]{Goo}
Sue Goodman.
\newblock Dehn surgery on {A}nosov flows.
\newblock In {\em Geometric dynamics ({R}io de {J}aneiro, 1981)}, volume 1007
  of {\em Lecture Notes in Math.}, pages 300--307. Springer, Berlin, 1983.

\bibitem[Hoz20a]{Hoz4}
Surena Hozoori.
\newblock Dynamics and topology of conformally {A}nosov contact 3-manifolds.
\newblock {\em Differential Geom. Appl.}, 73:101679, 16, 2020.

\bibitem[Hoz20b]{Hoz}
Surena Hozoori.
\newblock Symplectic {G}eometry of {A}nosov {F}lows in dimension 3 and
  {B}i-{C}ontact {T}opology.
\newblock Preprint arXiv:2009.02768, 2020.

\bibitem[Hoz21a]{Hoz2}
Surena Hozoori.
\newblock On {A}nosovity, divergence and bi-contact surgery.
\newblock Preprint arXiv:2109.00566, 2021.

\bibitem[Hoz21b]{Hoz3}
Surena Hozoori.
\newblock Ricci {C}urvature, {R}eeb {F}lows and {C}ontact 3-{M}anifolds.
\newblock {\em J. Geom. Anal.}, 31(11):10820--10845, 2021.

\bibitem[HPS77]{HM}
M.~W. Hirsch, C.~C. Pugh, and M.~Shub.
\newblock {\em Invariant manifolds}.
\newblock Lecture Notes in Mathematics, Vol. 583. Springer-Verlag, Berlin-New
  York, 1977.

\bibitem[HT80]{HaTh}
Michael Handel and William~P. Thurston.
\newblock Anosov flows on new three manifolds.
\newblock {\em Invent. Math.}, 59(2):95--103, 1980.

\bibitem[Kan97]{Ka}
Yutaka Kanda.
\newblock The classification of tight contact structures on the {$3$}-torus.
\newblock {\em Comm. Anal. Geom.}, 5(3):413--438, 1997.

\bibitem[KC22]{KLMM}
Thomas Massoni August~Moreno Kai~Cielieback, Oleg~Lazarev.
\newblock Floer theory of {A}nosov flows in dimension three.
\newblock 2022.

\bibitem[Mit95]{Mit}
Yoshihiko Mitsumatsu.
\newblock Anosov flows and non-{S}tein symplectic manifolds.
\newblock {\em Ann. Inst. Fourier (Grenoble)}, 45(5):1407--1421, 1995.

\bibitem[Mit02]{Mit2}
Yoshihiko Mitsumatsu.
\newblock Foliations and contact structures on 3-manifolds.
\newblock In {\em Foliations: geometry and dynamics ({W}arsaw, 2000)}, pages
  75--125. World Sci. Publ., River Edge, NJ, 2002.

\bibitem[Nod04]{Nod}
Takeo Noda.
\newblock Regular projectively {A}nosov flows with compact leaves.
\newblock {\em Ann. Inst. Fourier (Grenoble)}, 54(2):481--497, 2004.

\bibitem[NT02]{NoTs}
Takeo Noda and Takashi Tsuboi.
\newblock Regular projectively {A}nosov flows without compact leaves.
\newblock In {\em Foliations: geometry and dynamics ({W}arsaw, 2000)}, pages
  403--419. World Sci. Publ., River Edge, NJ, 2002.

\bibitem[PS09]{PuM}
Enrique~R. Pujals and Mart\'{\i}n Sambarino.
\newblock On the dynamics of dominated splitting.
\newblock {\em Ann. of Math. (2)}, 169(3):675--739, 2009.

\bibitem[Puj07]{Pu}
Enrique~R. Pujals.
\newblock From hyperbolicity to dominated splitting.
\newblock In {\em Partially hyperbolic dynamics, laminations, and
  {T}eichm\"{u}ller flow}, volume~51 of {\em Fields Inst. Commun.}, pages
  89--102. Amer. Math. Soc., Providence, RI, 2007.

\bibitem[Sal21]{FS1}
Federico Salmoiraghi.
\newblock Surgery on {A}nosov flows using bi-contact geometry.
\newblock Preprint, 2021.

\bibitem[Sha]{Sha}
Mario Shannon.
\newblock Dehn surgeries and smooth structures on $3$--dimensional transitive
  {A}nosov flows.
\newblock Preprint, 2020.

\bibitem[Thu83]{Th}
William Thurston.
\newblock The geometry and topology of 3-manifolds.
\newblock 1983.

\bibitem[Wan15]{Wan1}
Andy Wand.
\newblock Tightness is preserved by {L}egendrian surgery.
\newblock {\em Ann. of Math. (2)}, 182(2):723--738, 2015.

\bibitem[Wei91]{We}
Alan Weinstein.
\newblock Contact surgery and symplectic handlebodies.
\newblock {\em Hokkaido Math. J.}, 20(2):241--251, 1991.

\end{thebibliography}
